\documentclass[11pt]{article}
\usepackage{latexsym,amsmath,color,amsthm,amssymb,epsfig,graphicx,mathrsfs}
\usepackage{graphicx}
\usepackage{amssymb}
\usepackage[left=1in,top=1in,right=1in,bottom=1in]{geometry}
\usepackage[linktocpage=true]{hyperref}
\usepackage{setspace}
\usepackage{amssymb, amsmath, amsthm, graphicx,mathrsfs}

\def\qed{\hfill\ifhmode\unskip\nobreak\fi\quad\ifmmode\Box\else\hfill$\Box$\fi}
\def\ite#1{\hfill\break${}$\hbox to 50pt {\quad(#1)\hfill}}
\newtheorem{thm}{Theorem}

\newtheorem{lem}[thm]{Lemma}

\begin{document}

\title{\vspace{-0.5in}A stability version for a theorem of Erd\H{o}s on nonhamiltonian graphs}

\author{
{{Zolt\'an F\" uredi}}\thanks{
\footnotesize {Alfr\' ed R\' enyi Institute of Mathematics, Hungary
E-mail:  \texttt{furedi.zoltan@renyi.mta.hu}. 
Research supported in part by the Hungarian National Science Foundation OTKA 104343,
 by the Simons Foundation Collaboration Grant 317487,
and by the European Research Council Advanced Investigators Grant 267195.
}}
\and
{{Alexandr Kostochka}}\thanks{
\footnotesize {University of Illinois at Urbana--Champaign, Urbana, IL 61801
 and Sobolev Institute of Mathematics, Novosibirsk 630090, Russia. E-mail: \texttt {kostochk@math.uiuc.edu}.
 Research of this author
is supported in part by NSF grants  DMS-1266016 and DMS-1600592
and grants 15-01-05867 and 16-01-00499  of the Russian Foundation for Basic Research.
}}
\and{{Ruth Luo}}\thanks{University of Illinois at Urbana--Champaign, Urbana, IL 61801, USA. E-mail: {\tt ruthluo2@illinois.edu}.}}

\date{August 23, 2016}

\maketitle

\vspace{-0.3in}

\begin{abstract} Let $n, d$ be integers with $1 \leq d \leq \left \lfloor \frac{n-1}{2} \right \rfloor$, and set $h(n,d):={n-d \choose 2} + d^2$ and 
  $e(n,d):= \max\{h(n,d),h(n, \left \lfloor \frac{n-1}{2} \right \rfloor)\}$.
Because $h(n,d)$ is quadratic in $d$, there exists a $d_0(n)=(n/6)+O(1)$ such that
\begin{equation*}
   e(n,1)> e(n, 2)> \dots >e(n,d_0)=e(n, d_0+1)=\dots = e(n,\left \lfloor \frac{n-1}{2} \right \rfloor). 
  \end{equation*}
A theorem by Erd\H{o}s states that for $d\leq \left \lfloor \frac{n-1}{2} \right \rfloor$, any  $n$-vertex
 nonhamiltonian graph $G$ with minimum degree $\delta(G) \geq d$ has at most
 $e(n,d)$ edges, and for 
 $d \geq d_0(n)$ the unique sharpness example is simply the  graph $K_n-E(K_{\lceil (n+1)/2\rceil})$.
Erd\H{o}s also presented a sharpness example $H_{n,d}$  for each $1\leq d \leq d_0(n)$.

We show that if $d< d_0(n)$ and a $2$-connected, nonhamiltonian $n$-vertex graph $G$ with  $\delta(G) \geq d$ has more than
$e(n,d+1)$ edges, then $G$ is a subgraph of $H_{n,d}$.
Note that $e(n,d) - e(n, d+1) = n - 3d - 2 \geq n/2$ whenever $d< d_0(n)-1$.

\end{abstract}

\begin{flushright}
 Dedicated to the memory of Professor H. Sachs.
\end{flushright}

\section{Introduction}
We use standard notation. In particular, $V(G)$ denotes the vertex set of a graph $G$, $E(G)$ denotes
the edge set of $G$, and $e(G)=|E(G)|$. Also, if $v\in V(G)$, then $N(v)$ denotes the neighborhood of $v$ and
$d(v)=|N(v)|$.
  Ore~\cite{Ore} proved the following
Tur\' an-type result:

\begin{thm}[Ore~\cite{Ore}]\label{Ore}
If $G$ is a nonhamiltonian graph on $n$ vertices, then $e(G) \leq {n-1 \choose 2} + 1$.
\end{thm}

This bound is achieved only for the $n$-vertex graph obtained from the complete graph $K_{n-1}$ by adding a vertex of degree $1$.
Erd\H{o}s ~\cite{Erdos}  refined the bound in terms of the minimum degree of the graph:

\begin{thm}[Erd\H{o}s~\cite{Erdos}]\label{Erdos} Let $n, d$ be integers with $1 \leq d \leq \left \lfloor \frac{n-1}{2} \right \rfloor$, and set $h(n,d):={n-d \choose 2} + d^2$.
If $G$ is a nonhamiltonian graph on $n$ vertices with minimum degree $\delta(G) \geq d$, then
     \[e(G) \leq \max\left\{ h(n,d),h(n, \left \lfloor \frac{n-1}{2} \right \rfloor)\right\}=:e(n,d).\]
This bound is sharp for all $1\leq d\leq \left \lfloor \frac{n-1}{2} \right \rfloor$.
\end{thm}

To show the sharpness of the bound, for $n,d \in \mathbb N$ with $d \leq \left \lfloor \frac{n-1}{2} \right \rfloor$,
consider the graph $H_{n,d}$  obtained from a copy of $K_{n-d}$, say with vertex set $A$,
 by adding $d$ vertices of degree $d$ each
of which is adjacent to the same $d$ vertices in $A$.
An example of $H_{11,3}$ is below.
\begin{figure}[!ht]
  \centering
    \includegraphics[width=0.25\textwidth]{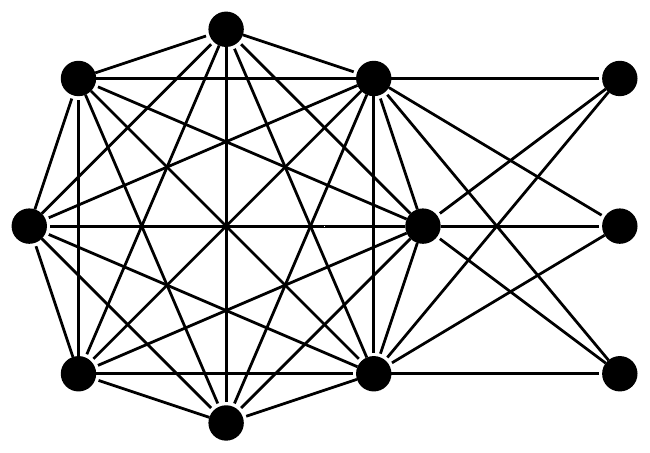} 
  \caption{$H_{11,3}$}
\end{figure}

By construction, $H_{n,d}$ has minimum degree $d$, is nonhamiltonian, and $e(H_{n,d}) = {n-d \choose 2} + d^2 = h(n,d)$.
Elementary calculation shows that $h(n,d)> h(n,\left \lfloor \frac{n-1}{2} \right \rfloor)$ in the range 
  $1\leq d\leq \left \lfloor \frac{n-1}{2} \right \rfloor$ if and only if 
  $d<(n+1)/6$ and $n$ is odd or  $d<(n+4)/6$ and $n$ is even.
Hence there exists a $d_0:=d_0(n)$ such that
\begin{equation*}
   e(n,1)> e(n, 2)> \dots >e(n,d_0)=e(n, d_0+1)=\dots = e(n,\left \lfloor \frac{n-1}{2} \right \rfloor),
  \end{equation*}
  where $d_0(n):= \left \lceil \frac{n+1}{6} \right \rceil$ if $n$ is odd, and 
  $d_0(n):= \left \lceil \frac{n+4}{6} \right \rceil$ if $n$ is even. 
Let $H'_{n,d}$ denote the graph that is an edge-disjoint union of two complete graphs
 $K_{n-d}$ and $K_{d+1}$ sharing one vertex. 

The result of this note is the following refinement of Theorem~\ref{Erdos}.

\begin{thm}\label{ma}
Let $n\geq 3$ and $d\leq \left \lfloor \frac{n-1}{2} \right \rfloor$. 
Suppose that $G$ is an $n$-vertex  nonhamiltonian graph  with minimum degree $\delta(G) \geq d$ such that 
\begin{equation}\label{equ2}
    e(G) > e(n,d+1)= \max\left\{ h(n,d+1), h(n, \left\lfloor \frac{n-1}{2}\right\rfloor)\right\}.
\end{equation}
(So we have $d< d_0(n)$.)
Then $G$ is a subgraph of either $H_{n,d}$ or $H'_{n,d}$. 
\end{thm}

This is a stability result in the sense that for $d<n/6$, 
each $2$-connected, nonhamilitonian $n$-vertex graph with minimum degree at least $d$
and ``close'' to $h(n,d)$ edges is a subgraph of the extremal graph $H_{n,d}$.
Note that $h(n,d) - h(n, d+1) = n - 3d - 2$ is at least $n/2$ for $d< d_0-1$. Note also that $e(H'_{n,d}) > e(n,d+1)$ only when $d = O(\sqrt{n})$. 

We will use the following well-known theorems of P\'osa.

\begin{thm}[P\'osa~\cite{Posa}]\label{Posa} Let $n\geq 3$.
If $G$ is a nonhamiltonian $n$-vertex graph, then there exists  $1\leq k \leq \left \lfloor \frac{n-1}{2} \right \rfloor$
such that $G$ has a set of $k$ vertices with degree at most $k$. \end{thm}

\begin{thm}[P\'osa~\cite{Po}]\label{Posa2} Let $n\geq 3$, $1\leq \ell <n$ and let $G$ be an $n$-vertex graph such that\\
$d(u)+d(v)\geq n+\ell\;$ for every non-edge $uv$ in $G$.
Then for every linear forest $F$ with $\ell$ edges contained in $G$, the graph $G$
has a hamiltonian cycle containing all edges of $F$.
\end{thm}

\section{Proof of Theorem~\ref{ma}}

Call a graph $G$  {\it saturated} if $G$ is nonhamiltonian but for each $uv \notin E(G)$, $G + uv$ has a hamiltonian cycle. 
Ore's proof \cite{Ore} of Dirac's Theorem~\cite{Dirac} yields that
\begin{equation}\label{D1}
\mbox{\em  for every $n$-vertex saturated graph $G$ and for each $uv \notin E(G)$, $d(u) + d(v) \leq n-1$.}
\end{equation}

First we show  two facts on saturated  graphs with many edges. 
 
\begin{lem}\label{g-d}Let $G$ be a saturated $n$-vertex  graph with $e(G) > h(n, \left \lfloor \frac{n-1}{2} \right \rfloor)$. 
Then for some $1\leq k \leq \left\lfloor \frac{n-1}{2} \right\rfloor$, $V(G)$ contains a subset $D$ of $k$ vertices of degree at most $k$ 
such that $G - D$ is a complete graph. 
\end{lem}
\begin{proof} Since $G$ is nonhamiltonian,  by Theorem~\ref{Posa}, there exists some $1\leq k \leq \left\lfloor \frac{n-1}{2} \right\rfloor$ such that $G$ has 
$k$ vertices with degree at most $k$. Pick the maximum such $k$, and let $D$ be the set of the vertices with degree at most $k$. 
Since $e(G) > h(n, \left \lfloor \frac{n-1}{2} \right \rfloor)$, $\; k<  \left \lfloor \frac{n-1}{2} \right \rfloor$. So, by the maximality of $k$,
$|D|=k$.

Suppose there exist $x, y \in V(G) - D$ such that $xy \notin E(G)$. Among all such pairs, choose $x$ and $y$ with the maximum $d(x)$.
Since $y\notin D$, $d(y)>k$. Let $D':=V(G)-N(x)-\{x\}$ and $k':=|D'|=n-1-d(x)$.
By~\eqref{D1}, 
\begin{equation}\label{818}
 \mbox{ \em $d(z) \leq n - 1 - d(x)=k'\;$ for all $\;z\in D'$. 
}
\end{equation}
So $D'$ is a set of $k'$ vertices of degree at most $k'$.
Since $y\in D'$, $\, k' \geq d(y) > k$. Thus by the maximality of $k$, we get
  $k' = n-1-d(x)> \left\lfloor \frac{n-1}{2}\right\rfloor$. 
Equivalently, $d(x) < \lceil \frac{n-1}{2}\rceil$. 
For all $z \in D'+\{x\}$, either $z \in D$ where $d(z)\leq k \leq \left\lfloor\frac{n-1}{2}\right\rfloor$, or $z \in V(G) - D$, 
and so $d(z) \leq d(x)\leq \left\lfloor \frac{n-1}{2}\right\rfloor$. 
It follows that $e(G) \leq h(n, \left\lfloor \frac{n-1}{2}\right\rfloor)$, a contradiction. 
\end{proof}

\begin{lem}\label{hnd} Under the conditions of Lemma \ref{g-d}, if $k = \delta(G)$, then $G = H_{n,\delta(G)}$ or $G = H'_{n,\delta(G)}$. 
\end{lem}
\begin{proof}Set $d:=\delta(G)$, and let $D$ be a set of $d$ vertices with degree at most $d$. 
Let $u\in D$. Since $\delta(G)\geq |D|=d$, $u$   has a neighbor $w\in V(G)-D$.
 Consider any  $v\in D - \{u\}$. 
By Lemma \ref{g-d}, $w$ is adjacent to all of $V(G) - D - \{w\}$. It also is adjacent  to $u$, therefore its degree is at least $n-d$.
We obtain 
    \[d(w) + d(v) \geq (n-d) + d = n.\] 
Then by~\eqref{D1}, $w$ is adjacent to $v$, and hence $w$ is adjacent to all vertices of $D$.

Let $W$ be the set of vertices in $V(G)-D$ having a neighbor in $D$. We have obtained that $W\neq \emptyset$ and 
\begin{equation}\label{D3}
\mbox{\em
 $N(u)\cap (V(G)-D)=W$ for all $u\in D$.}
\end{equation}
Let $G' = G[D \cup W]$. If $|W|=1$, then $G=H'_{n,d}$. 
If $|V(G')| = 2d$, then by~\eqref{D3}, each vertex $u \in D$ has the same $d$ neighbors in $V(G) -D$. 
Because $d(u) = d$, $D$ is an independent set. Thus $G=H_{n,d}$. Otherwise,  $d+2\leq |V(G')| \leq 2d - 1$, $|D|\geq 2$.

Fix a pair of vertices $w_1,w_2 \in W$. 
For any $x,y \in V(G')$,
    \[d(x) + d(y) \geq d + d \geq |V(G')| +1.\]
Therefore by Theorem~\ref{Posa2}, $G'$ has a hamiltonian cycle $C$ that uses the edge $w_1w_2$. 
Since $G'':=G-(V(G')-\{w_1,w_2\})$ is a complete graph, it contains a hamiltonian $w_1,w_2$-path $P$.  
Then $P \cup (C - w_1w_2)$ is a hamiltonian cycle of $G$, a contradiction. \end{proof}



{\it Proof of Theorem \ref{ma}.} Suppose that an $n$-vertex, nonhamiltonian graph $G$ satisfies the constraints of 
 Theorem~\ref{ma} for some $1\leq d \leq \left \lfloor \frac{n-1}{2} \right \rfloor$.  
We may assume $G$ is saturated, since if a graph containing $G$ is a subgraph of $H_{n,d}$ or $H'_{n,d}$, then $G$ is as well.



By Lemma \ref{g-d}, $G$ has a set $D$ of $k \leq \left \lfloor \frac{n-1}{2} \right \rfloor$ vertices with degree at most $k$ such that $G-D$ is a complete graph. Therefore $e(G) \leq {n-k \choose 2} + k^2 = h(n,k)$. If $k > d$, then $e(G) \leq \max\{ h(n,d+1), h(n, \left\lfloor \frac{n-1}{2}\right\rfloor)\}=e(n,d+1)$, a contradiction. Thus $k\leq d$. 
Furthermore, $k\geq \delta(G)\geq d$, and hence $k = d$. 
Also, since $e(G) >  h(n, \left\lfloor \frac{n-1}{2}\right\rfloor))$, we have $d+1 \leq d_0(n)\leq (n+8)/6$. Applying Lemma \ref{hnd} completes the proof.
\qed

\vspace{5mm}

{\bf Acknowledgment.} We thank both referees for their helpful comments.

{\bf Acknowledgment added on April 5, 2017.}
We have learned that Theorem~\ref{ma} has already been proved by Li and Ning as a lemma in~\cite{lining} with a somewhat
 different proof.

\end{document}